\documentclass[fleqn,runningheads]{svjour2}
\smartqed  
\usepackage{graphicx}
\usepackage{amsmath, amssymb, amsfonts}
\usepackage{xy}\input xy \xyoption{all}

\def\Z{\mathbb Z}
\def\C{\mathbb C}
\def\bP{\mathbb P}

\def\L{\mathcal L}
\def\H{\mathcal H}
\def\M{\mathcal M}
\def\N{\mathcal N}
\def\P{\mathcal P}
\def\X{\mathcal X}
\def\Y{\mathcal Y}

\def\l{\lambda}
\def\s{\sigma}
\def\a{\alpha}

\def\g{\gamma}
\def\d{\delta}

\def\f{\Phi}
\def\nf{\Psi}
\def\df{\Upsilon}
\def\p{\mathfrak p}

\def\iso{\cong}
\def\o{\otimes}
\def\emb{\hookrightarrow}
\def\Aut{\mathrm{Aut}}
\def\bAut{\overline{\mathrm{Aut}}}
\def\G{\overline G}

\def\<{\langle}
\def\>{\rangle}

\newenvironment{mat2}{\left(\begin{array}{cc}}{\end{array}\right)}

\journalname{Applicable Algebra in Engineering, Communication and Computing}
\begin{document}

\title{Hyperelliptic curves with reduced automorphism group $A_5$
}



\author{David Sevilla \and Tanush Shaska       
}


\institute{ D. Sevilla \at
Dpt. of Comp. Sci. and Software Eng., Concordia University\\
1455 de Maisonneuve W., Montreal QC, H3G 1M8 Canada\\
\email{sevillad@gmail.com}
           \and
T. Shaska \at
Department of Mathematics and Statistics, Oakland University\\
Rochester, Michigan 48309-4485, U.S.A.\\
\email{shaska@oakland.edu}           
}

\date{Received: date / Revised: date}

\maketitle

\begin{abstract}
We study genus $g$ hyperelliptic curves with reduced automorphism group $A_5$ and give equations $y^2=f(x)$ for
such curves in both cases where $f(x)$ is a decomposable polynomial in $x^2$ or $x^5$. For any fixed genus the
locus of such curves is a rational variety. We show that for every point in this locus the field of moduli is a
field of definition. Moreover, there exists a rational model $y^2=F(x)$ or $y^2=x F(x)$ of the curve over its
field of moduli where $F(x)$ can be chosen to be decomposable in $x^2$ or $x^5$. While similar equations have been
given in \cite{BCGG} over  $\mathbb R$, this is the first time that these equations are given over the field of
moduli of the curve.

\end{abstract}

\section{Introduction}
Let $\X_g$ denote a genus $g$ hyperelliptic curve defined over an algebraically closed field $k$ of characteristic
zero, $z_0$ its hyperelliptic involution, and $G:=\Aut (\X_g) $ its automorphism group. The group $\G = G/ \<z_0
\>$ is called the \emph{reduced automorphism} group of $\X_g$. We denote by $\H_g$ the moduli space of genus $g$
hyperelliptic curves and by $\L_g^G$ the locus in $\H_g$ of hyperelliptic curves with automorphism group $G$.

In previous works we have focused on the loci $\L_g^G$ of hyperelliptic curves with $V_4$ embedded in the
automorphism group $G$, or when $\G$ is isomorphic to $\Z_n$, or $A_4$; see \cite{GS}, \cite{Sh1}. This paper
continues on the same line of thought as \cite{Sh1} focusing instead on the case when $\G$ is isomorphic to $A_5$.

The second section covers basic facts on automorphism groups of hyperelliptic curves. The group $\G$ is a finite
subgroup of $PGL_2 (\C)$. By a theorem of Klein (see \cite{Kl}), $\G$ is isomorphic to one of the following: $
\Z_n, D_n, A_4, S_4, A_5$. We are interested in the latter case. We give a representation of the group $\G=A_5$ in
$PGL_2(\C)$. The group $A_5$ acts on the genus zero field $k(x)$ via the natural way. The fixed field is a genus 0
field, say $k(z)$. Thus, $z$ is a degree 60 rational function in $x$ which we denote by $z : =\phi(x)$. Using this
representation we compute the fixed field of $A_5$. This rational function $\phi (x)$ (up to a coordinate change)
can be decomposed in $x^2$, $x^3$, or $x^5$. Using computer algebra techniques (i.e, see \cite{Gu}) we compute
such decompositions and use the decomposition in $x^i$, $i=2,3,5$ to compute an equation $y^2 = f(x^i)$ of the
hyperelliptic curves. The equation for $i=2$ makes it possible to compute dihedral invariants of such curves (cf.
section 4).

In section three we determine the ramification signature $\s$ of the cover $\f : \X_g \to \X_g/\Aut(\X_g).$ Using
this ramification structure we are able to show that if $\bAut (X_g)\iso A_5$ then $g \equiv 0, 5, 9, 14, 15, 20,
24, 29 \pmod {30}$. Then the full automorphism group $\Aut (\X_g)$ is isomorphic to $\Z_2 \o A_5$ or $SL_2 (5)$.
Moduli spaces of covers $\f$ are Hurwitz spaces, which we denote by $\H_\s$. There is a map $\Phi_\s: \H_\s \to
\M_g$, where $\M_g$ is the moduli space of genus $g$ algebraic curves. For a fixed $g$ there is only one signature
that occurs for the cover $\f : \X_g \to \X_g/\Aut(\X_g)$. Hence, we denote by $\L_g$ the image $\Phi_\s (\H_\s)$
in the hyperelliptic locus $\H_g$. Given a curve $\X_g$ we would like to determine if it belongs to the locus
$\L_g$ and describe points $\p \in \L_g$. Hence, we need invariants which determine the isomorphism classes of
these curves. In the last part of section three we determine the parametric equations of such curves in all cases
$g \equiv 0, 5, 9, 14, 15, 20, 24, 29 \pmod {30}$. Using the decompositions of $\phi(x)$ we are able to compute
these equations $y^2=f(x)$ where $f(x)$ is a decomposable polynomial in $x^2$, $x^3$, or $x^5$.

In section four we give a brief introduction of the classical invariants of binary forms. Such invariants classify
the orbits of the $SL_2 (k)$-action on the space of binary forms. We use transvections to discover invariants
which give necessary conditions for a curve to have reduced automorphism group isomorphic to $A_5$ or full
automorphism group isomorphic to $\Z_2 \o A_5$ or $SL_2 (5)$. Such conditions appear in the literature for the
first time. Further, we compute the dihedral invariants of such curves and determine the algebraic relations among
them.

In the last section we discuss the field of moduli versus the field of definition for hyperelliptic curves with
reduced automorphism group $A_5$. This is a problem of algebraic geometry that goes back to Weil and Grothendieck.
It follows from \cite{Sh5} or \cite{GS} that for hyperelliptic curves with reduced automorphism group $A_5$ the
field of moduli is a field of definition. However, no rational models of the curve over the field of moduli have
been known. We construct such models for all curves $\X_g$ with $\bAut (\X_g) \iso A_5$. In the last part of the
paper we discuss in more detail the 1-dimensional families for all cases $g \equiv 0, 5, 9, 14, 15, 20, 24, 29
\pmod {30}$. In these cases we prove computationally that for such loci $\L_g$ we have $k(\L_g) = k(\l)$, where
$\l$ is the fourth branch point of the cover $\f: \X_g \to \bP^1$.

There is plenty of literature on the automorphism groups of hyperelliptic curves. Among many papers we mention
\cite{BS}, \cite{BCGG}, \cite{GS}, \cite{GSS}, \cite{Sh1}, \cite{Sh5}. Most of these papers have studied
determining the automorphism groups of the hyperelliptic curve. The main focus of this paper is the locus of
hyperelliptic curves with reduced automorphism group isomorphic to $A_5$ as a subvariety of the hyperelliptic
moduli and the field of definition versus the field of moduli for curves in this locus.\\

\noindent \textbf{Notation:} Throughout this paper $k$ denotes an algebraically closed field of characteristic
zero, $g$ an integer $\geq 2$, and $\X_g$ a hyperelliptic curve of genus $g$. $\M_g$ (resp., $\H_g$) is the moduli
space of curves (resp., hyperelliptic curves) defined over $k$. The symbol $(m)^r$ denotes a permutation which is
conjugate in $S_n$ to an $r$ product of $m$-cycles.

\section{Preliminaries}
Let $\X_g$ be a genus $g$ hyperelliptic curve defined over an algebraically closed field $k$ of characteristic
zero. We take the equation of $\X_g$ to be $y^2=F(x)$, where $\deg(F)=2g+2$. Denote the function field of $\X_g$
by $K:=k(x,y)$. We identify the places of $k(x)$ with the points of $\bP^1= k \cup \{\infty\}$ in the natural way.
Then, $K$ is a quadratic extension field of $k(x)$ ramified exactly at $n=2g+2$ places $\a_1, \dots , \a_n$ of
$k(x)$. The corresponding places of $K$ are called the {\it Weierstrass points} of $K$. Let $\P:=\{ \a_1, \dots ,
\a_n \}$. Thus, $K=k(x,y)$, where $y^2=\prod_{\a\in \P} (x-\a)$ and $\a \neq \infty$.

Let $G=\Aut(K/k)$. Since $k(x)$ is the only genus 0 subfield of degree 2 of $K$, then $G$ fixes $k(x)$. Thus,
$\mathrm{Gal}(K/k(x))=\< z_0 \>$, with $z_0^2=1$, is central in $G$. We call the \emph{reduced automorphism group}
of $K$ the group $\G:=G/\< z_0 \>$. Then, $\G$ is naturally isomorphic to the subgroup of $\Aut(k(x)/k)$ induced
by $G$. We have a natural isomorphism $\Gamma:=PGL_2(k) {\overset \iso \to } \Aut(k(x)/k)$. The action of $\Gamma$
on the places of $k(x)$ corresponds under the above identification to the usual action on $\bP^1$ by fractional
linear transformations $t \mapsto \frac {at+b} {ct+d}$. Further, $G$ permutes $\a_1, \dots , \a_n$. This yields an
embedding $\G \emb S_n$.

Because $K$ is the unique degree 2 extension of $k(x)$ ramified exactly at $\a_1,\dots,\a_n$, each automorphism of
$k(x)$ permuting these $n$ places extends to an automorphism of $K$. Thus, $\G$ is the stabilizer in
$\Aut(k(x)/k)$ of the set $\P$. Hence under the isomorphism $\Gamma \mapsto \Aut(k(x)/k)$, $\G$ corresponds to the
stabilizer $\Gamma_\P$ in $\Gamma$ of the $n$-set $\P$.

By a theorem of Klein, $\G$ is isomorphic to one of the following: $\Z_n$, $D_n$, $A_4$, $S_4$ or $A_5$. We are
interested in the latter case. The branching indices of the corresponding cover $\phi: \bP^1 \to \bP^1/ A_5$ are
$2, 5, 3$ respectively; see \cite{Sh1} for details of the general setup. That means that $A_5$ is given as $A_5
\iso \<\s_1, \s_2, \s_3\>$ where $\s_1 \s_2 \s_3=1$ and $\s_1, \s_2, \s_3$ have orders 2, 5, 3. How $\s_1, \s_2,
\s_3$ lift in the extension of $A_5$ will determine $G$. In the next section, we will determine the cover $\phi:
\bP^1 \to \bP^1$ explicitly. The lifting of the elements will determine the group and the equation of the
hyperelliptic curve.

Let $ \s_1= \begin{mat2} w & 1 \\ 1 & -w \\ \end{mat2}$ and
$\s_2 = \begin{mat2} \epsilon^2 & 0 \\ 0 & 1 \\
\end{mat2}$,
where $\omega=\frac{-1+\sqrt{5}}{2}$ and $\epsilon$ is a primitive $5^{th}$ root of unity. Then $\s_1, \s_2$ have
orders 2 and 5 respectively and $\s_3 = ( \s_1 \s_2)^{-1}$ has order 3. This gives an embedding of $A_5$ in $PGL_2
(\C)$ in the following way: $ A_5 \iso \< \s_1, \s_2 \> \emb PGL_2 (k) $. In the next section we will find the
fixed field $L$ of $k(x)$ under the $A_5$ action and study intermediate fields of the extension $k(x)/L$.

The group $A_5$ given above acts on $k(x)$ via the natural way. The fixed field is a genus 0 field, say $k(z)$.
Thus, $z$ is a degree 60 rational function in $x$, say $z=\phi(x)$. In this section we determine $\phi(x)$ and its
decompositions.
\begin{lemma}
Let $H$ be a finite subgroup of $PGL_2(k)$. Let us identify each element of $H$ with the corresponding Moebius
transformation and let $s_i$ be the $i$-th elementary symmetric polynomial in the elements of $H$,
$i=1,\ldots,|H|$. Then any non-constant $s_i$ generates $k(z)$.
\end{lemma}
\begin{proof}
It is easy to check that the $s_i$ are the coefficients of the minimum polynomial of $x$ over $k(z)$. It is
well-known that any non-constant coefficient of this polynomial generates the field. \qed
\end{proof}

\begin{corollary}
The fixed field of $A_5$ is generated by the function
\[z= -  \frac {\left(x^{20}-228x^{15}+494x^{10}+228x^5+1 \right)^3} {x^5 \, \left(x^{10} + 11x^5-1\right)^5}.\]
\end{corollary}

\begin{proof}
Apply the theorem to the embedding of $A_5$ given above. \qed
\end{proof}

The branch points of $\phi : \bP^1 \to \bP^1 \ $ are 0, 1728 and $\infty$. These correspond respectively to the
elements $\s_1, \s_2, \s_3$ in the monodromy group (cf. Section \ref{Hur}).  At the place $z=1728$ the function
has the following ramification:
\[\phi (x) - 1728= - \frac {\left( x^{30}+522x^{25}-10005x^{20}-10005x^{10}-522x^5+1\right)^2 }
{ x^5 \, \left(x^{10} + 11x^5-1\right)^5}.\]
We denote the following by
\begin{equation*}
\begin{split}
R & = x^{20}-228 x^{15}+494x^{10}+228x^5+1 \\
S & = x \ (x^{10} + 11x^5 -1)\\
T & =  x^{30}+522x^{25}-10005x^{20}-10005x^{10}-522x^5+1.\\
\end{split}
\end{equation*}
As we will see in the next section these functions will be instrumental in determining the equation of the
hyperelliptic curves.

\subsection{Decomposition of $\phi(x)$}\label{subsect22}
The automorphism group of $k(x)/k(\phi)$ is the embedding of $A_5$ detailed before. As $|A_5|=[k(x):k(\phi)]$,
there is a degree-preserving correspondence between subgroups of $A_5$ and intermediate fields in the extension.
By L\"uroth's Theorem, each of those fields is $k(h)$ for some rational function $h$. Now, it is clear that, in
general, $k(f)\subset k(h) \Leftrightarrow f=g\circ h\mathrm{\ for\ some\ } g$. Thus, we can use computer algebra
techniques to find all the decompositions of $\phi$ and describe the lattice of intermediate fields.

It is clear from the expression of $\phi$ that there is a decomposition $\phi=g(x^5)$. This comes also from the
fact that the subgroup $\<\epsilon x\>$ of $A_5$ corresponds to the field generated by $x\cdot\epsilon
x\cdot\epsilon^2 x\cdot\epsilon^3 x\cdot\epsilon^4 x=x^5.$

It is also possible to find decompositions involving $x^2$ or $x^3$ for functions that are equivalent to $\phi$.
Namely, for any $\sigma\in PGL_2(k)$, a generator of the field fixed for the conjugate group $\sigma
A_5\sigma^{-1}$ is $\phi(\sigma^{-1})$. If $\sigma$ is chosen in such a way as having $\{x,-x\}<\sigma A_5
\sigma^{-1}$, then $k(x\cdot(-x))=k(x^2)$ will be an intermediate field by Lemma 1. This can be accomplished by
conjugating any involution of $A_5$ into $-x$. In the same manner, if an element of order 3 in $A_5$ is conjugated
into $\zeta_3 x$, where $\zeta_3$ is a primitive cubic root of $1$, the resulting function can be written in terms
of $x\cdot\zeta_3 x\cdot\zeta_3^2 x=x^3.$

We present the former case here, as it will be used later. The element $-1/x\in A_5$ satisfies
$\sigma\circ\frac{-1}{x}\circ\sigma^{-1}=-x,$ where $ \sigma=\frac{ix+1}{-ix+1}.$
Therefore, $\phi_1:=\phi(\sigma^{-1})$ will have $x^2$ as a component. Indeed,
\[\phi_1  =64\,\frac {{\bar R}^3} {\bar S^5}, \quad \phi_1 - 1728 = 256 (i+2) \ \frac {\bar T^2} {\bar S^5}\]
where
\begin{small}
\begin{equation*}
\begin{split}
\bar R = & (25\,{x}^{8}- \left( 210-280\,i \right) {x}^{4}-7-24\,i) \cdot (15\,{x}^{4}+ \left( 10+20\,i \right)
{x}^{2}-9+12\,i) \\
& ( 25\,{x}^{8}+ \left( 300+600\,i \right) {x}^{6}+ \left( 1110-1480\,i \right) {x}^{4}- \left( 660+120\,i
\right){x}^{2}-7-24\,i )\\
\bar S = & (x^2-1) \ (5x^2+3-4i) \ ( 25\,{x}^{8}- ( 100+200\,i ) {x}^{6}+ ( 630-840\,i
) {x}^{4} \\
& + ( 220+40\,i ) {x}^{2}-7-24\,i ) \\
\bar T  = & \, x \cdot (5 x^4+10 x^2+1)\ (x^4+10 x^2+5)\ (125 x^4-150 x^2+200 i x^2-7-24 i) \\
& (5 x^4\!-30 x^2\!+40 i x^2\!-7\!-24 i)\ (5 x^4\!+3\!-4 i)\ (5 x^4\!-10 x^2\!-20 i x^2\!-27\!+36 i)\\
& (45 x^4-10 x^2-20 i x^2-3+4 i).
\end{split}
\end{equation*}
\end{small}

\section{Automorphism groups and the corresponding loci}
In this section we determine the automorphism group of $\X_g$ and the ramification structure of the cover $\X_g
\to \X_g/\Aut(\X_g)$. Further, we will discuss the locus of such curves in the variety of moduli.

The automorphism group $G$ of the hyperelliptic curve is a degree 2 central extension of $A_5$.  The following
lemma is proved in \cite{GS}.
\begin{lemma}
Let $p\geq 2$, $\a \in G$ and $\bar \a$ its image in $\G$ with order $| \, {\bar \a} \, |= p $. Then,

i) $| \, \a \, | = p$ if and only if it fixes no Weierstrass points.

ii) $| \, \a \, | = 2p$ if and only if it fixes some Weierstrass point.
\end{lemma}

Thus, $\G$ is the monodromy group of a cover $\phi: \bP^1 \to \bP^1$ with signature $(\s_1, \s_2, \s_3)$ as in
section 2; see \S 2 in \cite{Sh1} for further details. We fix the coordinates in $\bP^1$ as $x$ and $z$
respectively and from now on denote the cover $\phi : \bP^1_x \to \bP^1_z$. Thus, $z$ is a rational function in
$x$ of degree $|\G|$. We denote by $q_1, q_2, q_3 $ the corresponding branch points of $\phi$. Let $S$ be the set
of branch points of $\f: \X_g \to \bP^1$. Clearly $q_1, q_2, q_3 \in S$. Let $W$ denote the images in $\bP^1_x$ of
Weierstrass places of $\X_g$ and $V:=\cup_{i=1}^3 \phi^{-1} (q_i)$.

Let $z= \frac {\nf (x) } { \df (x) }$, where $\nf, \df \in k[x]$. For each branch point $q_i$, $i=1,2,3$ we have
the degree $|\G|$ equation $z\cdot \df(x) - q_i \cdot \df (x) =\nf (x),$ where the multiplicity of the roots
correspond to the ramification index for each $q_i$ (i.e., the index of the normalizer in $\G$ of $\s_i$). We
denote the ramification of $\phi: \bP^1_x \to \bP^1_z$, by $\varphi_m^r, \chi_n^s,  \psi_p^t$, where the subscript
denotes the degree of the polynomial.

Let $\l \in S \setminus \{ q_1, q_2, q_3\}$. The points in the fiber of a non-branch point $\l$ are the roots of
the equation: $ \nf (x) -\l \cdot \df (x) =0.$ To determine the equation of the curve we simply need to determine
the Weierstrass points of the curve. For each
fixed $\phi$ there are the following cases:\\

\noindent 1) \,$V \cap W = \emptyset$,  \\
2) \,$V \cap W =\phi^{-1} (q_1)$, \\
3) \,$V \cap W =\phi^{-1} (q_2)$, \\
4) \,$V \cap W =\phi^{-1} (q_3)$, \\
5) \,$V \cap W =\phi^{-1} (q_1) \cup \phi^{-1} (q_2)$, \\
6) \,$V \cap W =\phi^{-1} (q_2) \cup \phi^{-1} (q_3)$, \\
7) \,$V \cap W =\phi^{-1} (q_1) \cup \phi^{-1} (q_3)$, \\
8) \,$V \cap W =\phi^{-1} (q_1) \cup \phi^{-1} (q_2) \cup \phi^{-1} (q_3)$.\\

From the above lemma we have that if the places in the fiber $\phi^{-1} (q_1)$, $\phi^{-1} (q_2)$, $\phi^{-1}
(q_3)$, are Weierstrass points then $\s_1, \s_2, \s_3$ lift in $G$ to elements of order 4, 6, and 10 respectively.
The first four cases give the group $\Z_2 \o A_5$ and the other four cases give the group $SL_2 (5)$. We have the
following table. The column containing the dimension $\d$ of the corresponding spaces will be explained in the
next subsection.
\begin{table}[h!] \label{table1}
      \caption{All possible signatures when the reduced automorphism group is $A_5$}
%
      \begin{center}
      \renewcommand{\arraystretch}{1.24}
      \begin{tabular}{||c|c|c|c|c|c||}
      \hline
      \hline
$\#$ & $G$ & $\G$ &  $\d$  &  $\f: \X_g \to \bP^1$  & $ \phi: \bP^1 \to \bP^1$    \\
      \hline       \hline
1 & $\Z_2\o A_5$ & & $\frac {g+1} {30}$& $(3^{40}, 5^{24}, 2^{60}, \dots , 2^{60})$&   \\
2 & $\Z_2\o A_5$ & & $\frac {g-5} {30}$& $(3^{40}, 10^{12}, 2^{60}, \dots , 2^{60})$&      \\
3 & $\Z_2\o A_5$ & & $\frac {g-15} {30}$& $(6^{20}, 10^{12}, 2^{60}, \dots , 2^{60})$&  \\
4 & $\Z_2\o A_5$ & $A_5$ & $\frac {g-9} {30}$& $(6^{20}, 5^{24}, 2^{60}, \dots , 2^{60})$&
$( 2^{30}, 3^{20}, 5^{12} )$   \\
5 & $SL_2(5)$& &$\frac {g-14} {30}$&$(4^{30}, 3^{40}, 5^{24}, 2^{60}, \dots , 2^{60})$&   \\
6 & $SL_2(5)$&  &$\frac {g-20} {30}$&$(4^{30}, 3^{40}, 10^{12}, 2^{60}, \dots , 2^{60})$&    \\
7 & $SL_2(5)$&  &$\frac {g-24} {30}$ &$(4^{30}, 6^{20}, 5^{24}, 2^{60}, \dots , 2^{60})$&   \\
8 &$SL_2(5)$&  &$\frac {g-30} {30}$ &$(4^{30}, 6^{20}, 10^{12}, 2^{60}, \dots , 2^{60})$&   \\
      \hline
      \hline
     \end{tabular}
     \end{center}
\end{table}
In the Table we give the ramification structure of $\f : \X_g \to \bP^1$. The tuple $(\s_1,\dots,\s_r)$
corresponding to this signature is such that $G \iso \< \s_1, \dots , \s_r\>$ and $\s_1 \cdots \s_r =1$. We call
this tuple $( \s_1, \dots, \s_r ) $ the \emph{signature tuple} of the covering (cf. Section~\ref{Hur} for
details).
\begin{corollary}
Let $\X_g$ be a genus $g\geq 2$ hyperelliptic curve with reduced automorphism group isomorphic to $A_5$. If $g$ is
odd then $\Aut (\X_g) \iso \Z_2 \o A_5$, otherwise $\Aut (\X_g) \iso SL_2 (5)$.
\end{corollary}

\subsection{Hurwitz spaces}\label{Hur}

In this section we  give a brief introduction to Hurwitz spaces. For further details the reader can check \cite{H}
among many other authors. Let $X$ be a curve of genus $g$ and $f: X \to \bP^1$ be a covering of degree $n$ with
$r$ branch points. We denote the branch points by $q_1,\ldots,q_r\in\bP^1$ and let $p\in
\bP^1\setminus\{q_1,\ldots,q_r\}$. Choose loops $\gamma_i$ around $q_i$ such that
\[\Gamma:=\pi_1 (\bP^1\setminus\{ q_1, \dots , q_r\},\ p)=\< \gamma_1, \ldots , \gamma_r\>,
\quad \gamma_1 \cdots \gamma_r=1.\]
$\Gamma$ acts on the fiber $f^{-1}(p)$ by path lifting, inducing a transitive subgroup $G$ of the symmetric group
$S_n$ (determined by $f$ up to conjugacy in $S_n$). It is called the \emph{monodromy group} of $f$. The images of
$\gamma_1,\ldots,\gamma_r$ in $S_n$ form a tuple of permutations $\s=(\s_1,\ldots,\s_r)$ called a tuple of
\emph{branch cycles} of $f$. We call such a tuple the \emph{signature} of $\phi$. The covering $f:X\to\bP^1$ is of
type $\s$ if it has $\s$ as tuple of branch cycles relative to some homotopy basis of $\bP^1\setminus \{ q_1,
\dots , q_r\}$.

Two coverings $f:X\to\bP^1$ and $f':X'\to\bP^1$ are \emph{weakly equivalent} (resp. \emph{equivalent}) if there is
a homeomorphism $h:X\to X'$ and an analytic automorphism $g$ of $\bP^1$ such that $g\circ f=f'\circ h$ (resp.,
$g=1$). Such classes are denoted by $[f]_w$ (resp., $[f]$). The \emph{Hurwitz space} $\H_\s$ (resp.,
\emph{symmetrized Hurwitz space $\H_\s^s$}) is the set of weak equivalence classes (resp., equivalence) of covers
of type $\s$, it carries a natural structure of an quasiprojective variety.

Let $C_i$ denote the conjugacy class of $\s_i$ in $G$ and $C=(C_1, \dots , C_r)$. The set of Nielsen classes
$\N(G, C)$ is
\[\N(G,\s):=\{(\s_1,\dots,\s_r)\ |\ \s_i\in C_i,\,G=\<\s_1,\dots,\s_r\>,\ \s_1\cdots\s_r=1\}\]
The braid group acts on $\N(G, C)$ as
\[[\g_i]: \quad (\s_1, \dots , \s_r) \to (\s_1, \, \dots , \, \s_{i-1}, \s_{i+1}^{ \s_i}, \s_i, \s_{i+2}, \dots, \s_r)\]
where $\s_{i+1}^{ \s_i}= \s_i \s_{i+1} \s_i^{-1}$. We have $\H\sigma=\H_\tau$ if and only if the tuples $\s$,
$\tau$ are in the same \emph{braid orbit} $\mathcal O_\tau = \mathcal O_\sigma$.

Let $\M_g$ be the moduli space of genus $g$ curves. We have morphisms $ \H_\s \overset {\Phi_\s} \longrightarrow
\H_\s^s \overset { \bar{\Phi}_\s }  \longrightarrow \M_g$ where $ [f]_w \to [f] \to [X]$.
Each component of $\H_\s$ has the same image in $\M_g$. We denote by $\L_g :={\bar \Phi}_\s (\H_\s^s).$ This
causes no confusion since for a fixed $g$ we are in one of the cases of Table 1.

Next, we see how this applies to our particular situation. The family of covers $\f: \X_g \to \bP^1$ as in Table
1, have monodromy group $\Z_2 \o A_5$ or $SL_2(5)$. We denote the set of branch points of $f$ by $S:=\{ q_1, \dots
, q_r \}$. The branch cycle description of $f\,$ is $(\s_1, \ldots, \s_r)$ as in Table 1. Since we have at least
$r-3$ branch points which have the same ramification then there is an action of $S_r$ permuting these branch
points (i.e., which correspond to the ramification type $(2)^{60}$). Notice that in case 1 there is an action of
$S_{r+1}$ on the set of branch points. The symmetrized Hurwitz space is birationally isomorphic to the locus of
hyperelliptic curves in hyperelliptic moduli $\H_g$ with reduced automorphism group $A_5$. It will be our goal to
determine this locus for any $g$. We summarize the results of this section in the next lemma.

\begin{lemma} Let $\X_g$ be a genus $g\geq 2$ hyperelliptic curve with reduced automorphism group isomorphic
to $A_5$ and $\L_g$ denote the locus of such curves in the hyperelliptic moduli $\H_g$. Then, $G:=\Aut(\X_g)$ and
the signature $\s$ of the covering $\f : \X_g \to \bP^1$ are given in Table 1. Further, each locus $\L_g$ is
$\d$-dimensional irreducible subvariety of the hyperelliptic moduli $\H_g$.
\end{lemma}

\begin{proof} The moduli dimension of these families of covers
is $\d = r-3$, where $r$ is the number of branch points of the cover $\f: \X_g \to \bP^1$. The ramification of
each branch point $q\in S \setminus \{ q_1, q_2, q_3\}$ is of the type $(2)^n$. The Hurwitz-Riemann formula
determines the number of branch points in each case. \qed
\end{proof}

\subsection{Parametrization of families}
In this section we state the equations of curves in each case of Table 1. Continuing with the notation of section
4.1 we have $W \subset \bigcup_{\l \in S \setminus \{ q_1, q_2, q_3\} } \phi^{-1} (\l).$
Thus the places of $W$ are roots of the polynomial
\[\Lambda(x) := \prod_{\l \in S \setminus \{ q_1, q_2, q_3\}} \left( \nf (x) -\l \cdot \df (x)    \right).\]
Then, the equation of the curve for all cases 1-8 is $y^2=f(x)$ where $f(x)$ is respectively
\begin{equation}
\label{norm_eq} \Lambda,\ \varphi \cdot \Lambda, \ \chi\cdot \Lambda,\ \psi\cdot \Lambda,\ \varphi \cdot \chi\cdot
\Lambda,\ \chi\cdot \psi\cdot \Lambda,\ \varphi \cdot \psi\cdot \Lambda,\ \varphi \cdot \chi \cdot \psi\cdot
\Lambda.
\end{equation}
Since we know $z = \frac {\nf (x) } { \df (x) }$ in each case, then it is an elementary exercise to compute the
equation of the curve for all cases of Table 1. In our case we can apply the above when $z=\phi (x)$ or $z=\phi_1
(x)$. In the first case we have
\begin{small}
\begin{equation*}
\begin{split}\scriptstyle
& \Lambda_i (x) = -x^{60}+(684 - \l_i) x^{55}-(55 \l_i + 157434) x^{50}-(1205 \l_i-12527460 ) x^{45} \\
& -(13090\l_i\!+\!77460495)x^{40}\!+\!(130689144\!-\!69585\l_i) x^{35}\!+\!(33211924\!-\!134761\l_i)x^{30} \\
& +(69585\l_i-130689144) x^{25}\!-(13090 \l_i\!+\!77460495)x^{20}\!-(12527460-1205\l_i) x^{15} \\
& -(157434+ 55 \l_i) x^{10}+( \l_i-684) x^5-1
\end{split}
\end{equation*}
\end{small}
Then, $\Lambda(x) = \prod_{i=1}^\d \Lambda_i (x).$ By replacing $\varphi, \chi, \psi$ with $R, S, T$ we determine
the equation of the curve in each case. In the second case we determine $\Lambda (x)$ using $z=\phi_1 (x)$ and
$\bar R, \bar S, \bar T$.

\section{Isomorphism classes of hyperelliptic curves with reduced automorphism group $A_5$}
In this section we discuss the invariants of hyperelliptic curves with reduced automorphism group $A_5$. Such
invariants are needed to describe the loci $\L^G_g$ and discuss the field of definition of such curves.  We will
consider the coefficients of our curves as variables in order to study the relations among the different function
fields that will be introduced.

To get a description of $\L_g^G$ for each case of Table 1, we need invariants which would classify the isomorphism
classes of hyperelliptic genus $g$ curves. These invariants are generators of the fixed field of $GL_2 (k)$ acting
on the $(d+1)$-dimensional space $V_d$ of binary forms of degree $d$.

We use the symbolic method of classical invariant theory to construct invariants of binary forms. Let $f(X, Y)$
and $g(X, Y)$ be binary forms of degree $n$ and $m$ respectively. We denote by $(f, g)^r$  their
\emph{$r$-transvection}; see  \cite{Sh1} for details. For the rest of this paper $F(X,Y)$ denotes a binary form of
degree $d:=2g+2$. Invariants (resp., covariants) of binary forms are denoted by $I_s$ (resp., $J_s$) where the
subscript $s$ denotes the degree (resp., the order). We define the following covariants and invariants:
\begin{equation*}
\begin{split}
\begin{aligned}
&J_{4j} := (F,F)^{d-2j}, \  j=1,\dots,g, \\
&I_4 := (J_{12}, J_{12})^{12}, \\
&I_6^* := ((F, J_{20})^{20}, (F, J_{20})^{20})^{d-20}.
\end{aligned}
\quad
\begin{aligned}
&I_2 := (F,F)^d, \\
&I_6 := ((F, J_{12})^{12}, (F, J_{12})^{12})^{d-12}, \\ \\
\end{aligned}
\end{split}
\end{equation*}

The $GL_2(k)$-invariants are called \emph{absolute invariants}. We define the following absolute invariants:
\[i_1:=\frac {I_4} {I_2^2}, \quad  i_2:=\frac {I_6} {I_2^3}, \quad i_3= \frac {I_6^*} {I_2^3},\quad i_4= \frac {I_6^2}
{I_4^3}.\]
We will only perform computations on subvarieties $\L_g \subset \H_g$ of dimension $\delta \leq 1$, hence don't
need other absolute invariants. Next we will give necessary conditions on these invariants for the corresponding
curve to have reduced automorphism group $A_5$ and full automorphism group $\Z_2 \o A_5$ or $SL_2(5)$.

\begin{lemma}
Let $\X_g$ be a hyperelliptic curve with genus $g\leq 60$ such that $\bAut(\X_g) \iso A_5$. Then the invariants
$(J_i,J_i)^i$ are zero for $i=4,8,16,28$.
\end{lemma}

\begin{proof}
In all cases, it can be directly computed that the corresponding $J_i$'s are zero. \qed
\end{proof}

Let $\X_g$ be a genus $g$ hyperelliptic curve such that $ \bAut (\X_g) \iso A_5$. Then, $\X_g$ is isomorphic to a
curve given by the equation $y^2= F(x^2)  \quad or \quad y^2= x \ F(x^2)$,
with
\[F(x)= x^{d} + a_{d-1} x^{d-1}+ \dots + a_1 x +1,\]
where $d\ = \ g+1 $ or $ g$. Such equation is called the \emph{normal equation} of the curve $\X_g$. The following
\[u_i:=  a_1^{d-i} \, a_i \, + \, a_{d-1}^{d-i} \, a_{d-i}, \quad for \quad 1 \leq i \leq d-1\]
are called \emph{dihedral invariants}. Assume $d= g+1$ (the other case is similar). From the definition of the
invariants we have
\begin{equation*}
\begin{split}
& u_i \ = \  a_1^{g+1-i} a_i + a_g^{g+1-i} a_{g+1-i}, \\
& u_{g+1-i}\ = \ a_1^i a_{g+1-i} + a_g^i a_i,
\end{split}
\end{equation*}
for each $2 \leq i \leq g-1$. Notice that solving this linear system we have $a_i \in k(u_1, \dots , u_g, a_1,
a_g)$, for each $2 \leq i \leq g-1$. For $u_1, u_g$ we have the equation
\begin{equation}\label{quad}
2^{g+1}\, a_g^{2g+2} - 2^{g+1}\, u_1 \, a_g^{g+1} + u_g^{g+1}=0
\end{equation}
which is a quadratic polynomial in $a_g^{g+1}$. It is shown in \cite{GS} that $\L_g$ is a rational variety and
$k(\L_g)=k(u_1, \dots , u_{d-1})$. The next theorem determines a relation between dihedral invariants.
\begin{theorem}\label{disc}
Let $\X_g $ be a genus $g$ hyperelliptic curve with $\bAut(\X_g)\iso A_5$ and $(u_1, \dots , u_g)$ its
corresponding dihedral invariants.  Then

i) If $g$ is odd then $\Aut(\X_g)\iso\Z_2\o A_5$ and $2^{\frac{d-2}2}\,u_1-u_{d-1}^{\frac d 2}=0.$

ii) If $g$ is even then $\Aut(\X_g)\iso SL_2(5)$ and $2^{\frac{d-2}2}\,u_1+u_{d-1}^{\frac d 2}=0.$
\end{theorem}

\begin{proof}
i) This follows from Theorem 3, i) in \cite{GS}.

ii) The equation of $\X_g$ is given by $ y^2= x\ F(x^2)$ where $F(x^2)$ is a polynomial of degree $d= g$ in $x^2$.
Computing invariants is the same as in part i). In this case the involutions of $\Aut(\X_g)$ lift to elements of
order 4 in $\Aut (\X_g)$. From Theorem 3, ii) in \cite{GS} we have the equation of part ii). This completes the
proof. \qed
\end{proof}

Since the discriminant of the quadratic in Eq.~\eqref{quad} is zero (see Thm.~\ref{disc}) we have $a_g^{g+1} =
\frac {u_1} 2$. Hence, $[k(a_1, \dots , a_g):k(u_1, \dots , u_g)]=g+1$. Let $\Y_g$ be a hyperelliptic curve with
reduced automorphism group $A_5$ and equation $y^2= b_{g+1} x^{2g+2} +  b_g x^{2g}+ \dots + b_1 x^2 + b_0.$
Since $u_1, \dots , u_g$ are invariants under any coordinate change then $k (b_0, \dots , b_{g+1})$ is an
extension of $k(u_1, \dots , u_g)$. This curve $\Y_g$ can be normalized by means of the transformation $(x,y) \to
\left(x\cdot\sqrt[2g+2]{\frac{b_0}{b_{g+1}}},\ y\cdot\sqrt{b_0}\right)$, which gives
\begin{scriptsize}
\[y^2 = x^{2g+2} + \frac{b_g}{b_0}\left(\frac{b_0}{b_{g+1}}\right)^\frac{2g}{2g+2} x^{2g}
 + \frac{b_{g-1}}{b_0}\left(\frac{b_0}{b_{g+1}}\right)^\frac{2g-2}{2g+2} x^{2g-2}
 + \cdots + \frac{b_1}{b_0}\left(\frac{b_0}{b_{g+1}}\right)^\frac{2}{2g+2}
 x^2 +1. \]
\end{scriptsize}

Notice that $a_g^{g+1} \in k (b_0, \dots , b_{g+1})$. Since all the other $a_i$'s can be expressed in terms of
$a_g$ then $[ k (a_1, \dots , a_g) : k (b_0, \dots , b_{g+1})]=g+1$ we have that $k(u_1, \dots , u_g)= k (b_0,
\dots , b_{g+1})$.

The cover $\f : \X_g \to \bP^1$ has $\d + 3$ branch points. Let $S\setminus\{q_1,q_2,q_3\}=\{\l_1,\dots,\l_\d\}$.
Then, the isomorphism class of the corresponding curve is determined up to permutation of $\l_1, \dots , \l_\d$.
Invariants of this action are the symmetric polynomials in $\l_1, \dots , \l_\d$. Hence, $\Lambda(x)$ has
coefficients in terms of the elementary symmetric polynomials $s_1, \dots , s_\d$ of $\l_1, \dots , \l_\d$. Thus,
$k( b_0, \dots , b_{g+1}) \subset k(s_1, \dots , s_\d)$. Since $k(\L_g) \subset k( b_0, \dots , b_{g+1})$ then
there are at least $\d$-independent $b_i$'s. Therefore, $s_1, \dots , s_\d$ can be expressed in terms of $b_0,
\dots , b_{g+1}$. Thus, $k( b_0, \dots , b_{g+1}) = k(s_1, \dots , s_\d)$. Thus, we have the following:

\begin{lemma}\label{lemma6}$ k( u_1, \dots , u_d)=k(s_1, \ldots , s_\d)$.
\end{lemma}

From the computational point of view, to express $s_1, \dots , s_\d$ as rational functions in terms of $u_1, \dots
, u_d$ one can proceed as follows. A quick inspection shows that the coefficients of
\[\Lambda_i (x) = x^{60}+ a_{29}\ x^{58} + \cdots + a_1\, x^2 + 1,\]
which are linear polynomials in $\l$, satisfy $a_i \cdot \epsilon_3^i = a_{30-i}$, for $i=1, \ldots ,14$, where
$\epsilon_3$ is the primitive cubic root of unity with negative imaginary part.

For the polynomial
\[\Lambda(x)= \prod_{i=1}^\d \Lambda_i(x)=x^{60\d}+A_{30\d-1}x^{60\d-2}+\cdots+A_1x^2+1,\]
each coefficient is symmetric in $\l_1,\ldots,\l_\d$; moreover, each $A_i$ is a linear polynomial in
$s_1,\ldots,s_\d$. Also,
\[A_i \cdot \epsilon_3^i = A_{30\d-i},\quad i=1,\ldots,15\d-1.\]
Applying these relations to the dihedral invariants (starting with the last one) we obtain:
\begin{equation*}
\begin{split}
&u_{30\d-1}= A_1 \, A_{30\d-1} \, + \, A_{30\d-1} \, A_1 = 2 \, \epsilon_3 \, A_1^2, \\
&u_{30\d-2}= A_1^2 \, A_{30\d-2} \, + \, A_{30\d-1}^2 \, A_2 = 2\,\epsilon_3^2\,A_1^2\,A_2, \\
&\hspace{8.5em}\ldots \\
&u_{30\d-i}=  A_1^i \, A_{30\d-i} \, + \, A_{30\d-1}^i \, A_i = 2\,\epsilon_3^i\,A_1^i\,A_i. \\
\end{split}
\end{equation*}
Since, $u_i=a_1^{d-i} \, a_i \, + \, a_{d-1}^{d-i} \, a_{d-i}$ for  $1 \leq i \leq d-1,$ then combining the first
equation and the ones for even values of $i$, we obtain equalities
$A_i=\frac{u_{30\d-i}}{(\epsilon_3\,u_{30\d-1})^{i/2}}$ and as each $A_i$ is a linear polynomial in
$s_1,\ldots,s_\d$, this provides a linear system of equations, from which we can express each $s_j$ as a rational
function in $u_1,\ldots,u_{30\d-1}$.

\subsection{One-dimensional families}
Next we describe explicitly the 1-dimensional loci. We find the equations of such loci in terms of invariants
$i_1, i_2$. Further, we give a computational proof of the above lemma. As an example, we will compute dihedral
invariants in the case that $g \equiv 29 \pmod {30}$ and $\d=1$.

Let us denote the only parameter as $\lambda$. In this case, $g=29$ and $y^2=\Lambda(x)$ with $\Lambda$ defined
above. Then

\[u_1=2^{31}\cdot 5^{15}\frac{(11\,\lambda+32832)^{30}}{(\lambda-1728)^{30}},\quad u_{29} = \ 2^3
\cdot 5 \, \frac{(11\,\lambda+32832)^2}{(\lambda-1728)^2},\]
\[u_i=\frac{(\alpha_i\lambda+\beta_i)(11\,\lambda+32832)^{30-i}}{(\lambda-1728)^{31-i}},\quad i =
2,\ldots,28\,,\ \alpha_i,\beta_i\in\Z.\]
It is easily checked that $2^{14} u_1 - u_{29}^{15}\ = \ 0$, as expected from above.

Also, any generator of the field $k(u_1,\ldots,u_{29})$ is a rational function in $\lambda$ whose degree divides
those of $u_i$. As the degrees of $u_{28}$ and $u_{29}$ are 3 and 2 respectively, without loss of generality we
can choose any degree one rational function in $\lambda$, that is, $k(u_1,\ldots,u_{29})=k(\lambda)$ as expected.
The next lemma gives a computational proof of this result for all 1-dimensional loci. From the proof of the
following lemma we get an explicit expression on $\l$ in terms of $i_1, i_2$. Such expression will be used in the
next section.
\begin{lemma}\label{prop1}
For each of the 1-dimensional $\L_g $ we have $k(\L_g ) = k(\l)$.
\end{lemma}
\begin{proof}
The invariants are
\begin{tiny}
\begin{equation*}
\begin{split}
i_1 & \ = \ \frac{1948908}{7397845567} \frac{(-15159961555740000-610337874000\lambda+791091587\lambda^2)^2}{(11586093746490000+872196589\lambda^2-931385301000\lambda)^2}, \\
i_2 & =  \frac{4947228\ (79290599\lambda-42335695500)^2(-15159961555740000-610337874000\lambda+791091587\lambda^2)^2}{1083437009726901515\ (11586093746490000+872196589\lambda^2-931385301000\lambda)^3}. \\
\end{split}
\end{equation*}
\end{tiny}

By L\"uroth's Theorem there is a rational function in $\l$ that generates $k(i_1,i_2)$. By computing this
generator (see for example \cite{Gu2}) it is proved that $\lambda$ is a generator of $k(i_1,i_2)$. The other cases
are proved in the same way. \qed
\end{proof}

By eliminating $\l$ we can explicitly find the curve
\begin{equation}\label{eq_1_dim}
F(i_1, i_2)=0
\end{equation}
for each of the eight cases. For example, the equation of the curve in the first case is given by
\begin{tiny}
\begin{equation*}
\begin{split}
20104543529222176607891970551365425625i_2^4-6001516794980613854767134781868434500i_2^3 i_1 \\
+671664430878843510918689481392772150 i_2^2 i_1^2-467825523547842914848108169841758572200i_1^3i_2^2\\
-33400604375309785622232551775685380 i_1^3 i_2+69851513504555488123050532974625671120 i_1^4 i_2\\
+622702796403678565883409475309881 i_1^4-2609325640118276782171286285338389288 i_1^5\\
+2733479091269756882118693138958399101456 i_1^6 &=0
\end{split}
\end{equation*}
\end{tiny}
In each case this is a genus 0 curve with degrees 6 and 4 in $i_1$ and $i_2$ respectively, hence these curves have
singular points. In each case there are exactly three singular points and for each singular point $(i_1, i_2)$
there are two corresponding values of $\l$.

We determine these points explicitly and present the quadratic equations in $\l$ whose roots determine those
points. For each case the second and third singular point corresponds to $(i_1, i_2)=(0,0)$ and the point at
infinity given by $I_2=0$.

\begin{table}[ht!] \label{table2}
      \caption{Singular points of 1-dimensional loci}
%
\begin{center}
\begin{tiny}
\begin{tabular}{||c|c||}
      \hline
      \hline
$\#$ & Values of $\l$   \\
      \hline
      \hline
 & $452144735218242469277017\l^2-482828029389149632341153000\l-8593063274412012696185238840000$   \\
1 & $791091587\l^2-610337874000\l-15159961555740000$   \\
 & $872196589\l^2-931385301000\l+11586093746490000$  \\
      \hline
 & $45116739209875087720855199586628\l^2-43556541042223826596073807880918300\l$   \\
 & $-138531873472176494963183499855707291875$   \\
2 & $1651853764\l^2-1226334498300\l-5257525430501875$   \\
 & $5700085544\l^2-5799389184300\l+29819258427080625$   \\
      \hline
 & $25920118616911092183126784613617142\l^2-44031520362372236593738424989592507950\l$   \\
 & $-380728179705646173243900805566261020741875$   \\
3 & $1123023677098\l^2-1632357832704050\l-16735382221690758125$   \\
 & $11130104653\l^2-19337115814550\l+144111607427085625$   \\
      \hline
 & $1190289762560291723133371786217787\l^2-2109180129399825416728336502192357000\l$   \\
 & $-102603051826076134426802508773908422000000$   \\
4 & $31930385620603\l^2-46875776542808000\l-2760999374275855500000$   \\
 & $6105542623\l^2-10818953153000\l+230260862893387500$   \\
      \hline
 & $6333690499915638419332937497733\l^2-2171594295704055460635952732129500\l$   \\
 & $-388162393043218880265390321313068639375$   \\
5 & $2357171794013\l^2-531948616761375\l-144507437760700783125$   \\
 & $1639203229\l^2-562023733500\l+54739184825587500$   \\
      \hline
 & $6548345875574794801166675435529962539362\l^2-2054584724746639344314578064291669769060100\l$   \\
 & $-66477004559491595548969707908580013974645529375$   \\
6 & $17638004386446978\l^2-3958284234892272700\l-179314085452120858954375$   \\
 & $29883184652\l^2-9770219914300\l+286336970555605625$  \\
      \hline
 & $5197143015623358421052917257787762\l^2-4064516960859449385634468871138790750\l$   \\
 & $-1973346971902336126577309580519879740484375$   \\
7 & $167241649141649\l^2-87194298622737750\l-63518438366227732921875$  \\
 & $7883609626\l^2-6165515349750\l+932376249595828125$   \\
      \hline
 & $50049608492559474153964972804988742465553\l^2-36990019047097223907490687861759806337326200\l$   \\
 & $-2950939252953188574421512956195380140260947773125$   \\
8 & $14303777741547512\l^2-7805526971818231735\l-844416597642584618857750$   \\
 & $4237002269\l^2-3224689016170\l+134559773845245500$   \\
      \hline
      \hline
\end{tabular}
\end{tiny}
\end{center}
\end{table}

\section{Rational models over the field of moduli}
In this section we study the field of moduli of hyperelliptic curves with reduced automorphism group $A_5$. Let
$\X$ be a curve defined over $\C$. A field $F \subset \C$ is called a \emph{field of definition} of $\X$ if there
exists $\X'$ defined over $F$ such that $\X'$ is isomorphic to $\X$ over $\C$.

The \emph{field of moduli} of $\X$ is a subfield $F \subset \C$ such that for every automorphism $\sigma \in \Aut
(\C)$ the following holds: $\X$ is isomorphic to $\X^\sigma$ if and only if $\, \, \sigma_F = id$. We will use
$\p=[\X]\in \M_g$ to denote the corresponding \emph{moduli point} and $\M_g (\p)$ the residue field of $\p$ in
$\M_g$. The field of moduli of $\X$ coincides with the residue field $\M_g (\p)$ of the point $\p$ in $\M_g$. The
notation $\M_g (\p)$ (resp., $M(\X)$) will be used to denote the field of moduli of $\p \in \M_g$ (resp., $\X$).
If there is a curve $\X^\prime$ isomorphic to $\X$ and defined over $M(\X)$, we say that $\X$ has a \emph{rational
model over its field of moduli}. As mentioned above, the field of moduli of curves is not necessarily a field of
definition.

Let $\X_g$ be a genus $g$ hyperelliptic curve with reduced automorphism group isomorphic to $A_5$. Then its field
of moduli is a field of definition. Next, we give a rational model of the curve over the field of moduli.

\begin{theorem}
Let $\p \in \H_g$ such that $\bAut (\p) \iso A_5$ and $u_1, \dots , u_s$, $s=g$ or $ g-1$, the corresponding
dihedral invariants. Then, there is a rational model over the
field of moduli  $M(\p)$  given in the form $y^2=F(x)$ or $y^2= x\,F(x)$  where:\\

a) $F(x)$ can be chosen to be decomposable polynomial in $x^2$ as below \\

\qquad i) if $\Aut (\p) \iso \Z_2 \o A_5$ then
            \begin{equation*} y^2=u_1 x^{2g+2} + u_1 x^{2g} + u_2 x^{2g-2} + u_3 x^{2g-4}+ \dots + u_g x^2+2;
            \end{equation*}

\qquad ii) if $\Aut (\p) \iso SL_2(5)$ then
            \begin{equation*}
            y^2= x( u_1 x^{2g} + u_1 x^{2g-2} + u_2 x^{2g-4} + \dots + u_{g-1} x^2 +2);
            \end{equation*}

b) $F(x)$ can be chosen to be decomposable in $x^5$ as in Eq.~\eqref{norm_eq}.
\end{theorem}

\begin{proof}
a) Let $[\X_g]=\p \in \H_g$ such that $\bAut (\p) \iso A_5$. If $g$ is odd then dihedral invariants are $u_1,
\dots , u_g$, otherwise they are $u_1, \dots , u_{g-1}$.

i) Let $\Aut (\p) \iso \Z_2 \o A_5$. Then $\X_g$ has normal equation
\[y^2 = x^{2g+2} + a_g x^{2g}+\dots + a_1 x^2 + 1.\]
From the definition of the invariants we have
\begin{equation}\label{system2x2}
\begin{split}
& u_i \ = \  a_1^{g+1-i} a_i + a_g^{g+1-i} a_{g+1-i}, \\
& u_{g+1-i}\ = \ a_1^i a_{g+1-i} + a_g^i a_i,
\end{split}
\end{equation}
for each $2 \leq i \leq g-1$.
For $u_1, u_g$ we have the equation
\begin{equation*}
2^{g+1}\, a_g^{2g+2} - 2^{g+1}\, u_1 \, a_g^{g+1} + u_g^{g+1}=0
\end{equation*}
which is a quadratic polynomial in $a_g^{g+1}$. Since the discriminant of this quadratic is zero (see Theorem
\ref{disc}) we obtain $a_g^{g+1} = u_1/2$. By the transformation $x \to \sqrt{a_g} \ x$, the curve is isomorphic
to a curve with equation
\[y^2 = \frac {u_1} 2 x^{2g+2} + \frac{u_1}2 x^{2g}+ a_{g-1} a_g^{g-1} x^{2g-2}+ \dots + a_1 a_g x^2 +1.\]

Now, it suffices to show that for $2 \leq i \leq g-1$ we have $a_i a_g^i = u_{g+1-i}/2$ which is equivalent  to
$a_i a_g^i = a_{g+1-i} a_1^i$ (by the definition of $u_{g+1-i}$).
This equality can easily be proven by solving the linear system in Eq.~\eqref{system2x2} for $a_i$ and $a_{g+1-i}$
and using these values in the previous equation.

ii) Let $\Aut (\p) \iso SL_2(5)$. Then $\X_g$ has normal equation
\[y^2 = x \left( x^{2g} + a_{g-1}x^{2g-2}+ \dots + a_1 x^2 + 1 \right).\]
The proof is similar to the above by replacing $g$ with ${g-1}$. The transformation $x \to \sqrt{a_{g-1}}$ fixes 0
and $\infty$ and the result follows.

Since the dihedral invariants are in the field of moduli $M(\p)$ it is enough to show that $\X_g$ is isomorphic to
a curve $C$ whose coefficients are in terms of such invariants.

Part b) follows from part a) and Lemma~\ref{lemma6}. \qed
\end{proof}

\begin{remark}
Similar  equations to the equations in Eq.~\eqref{norm_eq} are given also in \cite{BCGG}, as kindly pointed out by
the referee. However, there is no discussion in that paper of the rational model of the curves over the field of
moduli. The point of the above theorem is that such equations are defined over the field of moduli.
\end{remark}

As it has already been stated, the expressions in parts a) and b) of the previous Theorem define the same $\p$. We
will explicitly show this by giving the corresponding isomorphism. In Subsection \ref{subsect22} it was determined
that the inverse of the transformation $\sigma: \,  x \to \frac{ix+1}{-ix+1}$ transforms $\phi$ into $\phi_1$.
Therefore, the isomorphism given by  $(x,y) \to \left(\frac{x-1}{ix+i},\frac{y}{(ix+i)^{g+1}}\right)$,  transforms
the equation of part b) into an equation in $x^2$:
\[y^2= b_{g+1} x^{2g+2} +  b_g x^{2g}+ \dots + b_1 x^2 + b_0.\]
This can be normalized by means of the isomorphism
\[(x,y) \to \left(x\cdot\sqrt[2g+2]{\frac{b_0}{b_{g+1}}},\ y\cdot\sqrt{b_0}\right)\]
which gives
\begin{scriptsize}
\[y^2 = x^{2g+2} + \frac{b_g}{b_0}\left(\frac{b_0}{b_{g+1}}\right)^\frac{2g}{2g+2} x^{2g}
 + \frac{b_{g-1}}{b_0}\left(\frac{b_0}{b_{g+1}}\right)^\frac{2g-2}{2g+2} x^{2g-2}
 + \cdots + \frac{b_1}{b_0}\left(\frac{b_0}{b_{g+1}}\right)^\frac{2}{2g+2}
 x^2 +1.\]
\end{scriptsize}
As in the proof of the previous Theorem, it is enough to compose this with $(x,y) \to \left(
x\cdot\sqrt{\frac{b_g}{b_0}}\left(\frac{b_0}{b_{g+1}}\right)^\frac{g}{2g+2}, \frac{y}{\sqrt{2}} \right)$ to obtain
the rational model as a polynomial in $x^2$.


\subsection{Computing the rational model}

\begin{table}\label{table3}
      \caption{Field of definition for the singular points}
      \begin{center}
\begin{scriptsize}
      \begin{tabular}{||c|c||}
      \hline
      \hline
$\#$ & $d$ such that $M (\p)= k(\sqrt{d})$  \\
      \hline
      \hline
 & 6594752841114090745134757 \\
1 & 127067509222 \\
 & -\,27468005002203037701 \\
      \hline
 & 741854166910125814698682452912604588627104323162904099673 \\
2 & 120950912295937 \\
 & -\,1318890572777620357905 \\
      \hline
 & 3877164163606363023773232119905718360665213621866915565002867565515 \\
3 & 7287697079146593051915 \\
 & -\,67133167127519339801955 \\
      \hline
 & 287030471019588726034132917522068933305 \\
4 & 931923194601696118570 \\
 & -\,550642030389053730301265 \\
      \hline
 & 77622682424472206764752607551443431 \\
5 & 13982754260355689869 \\
 & -\,3984430259985622758510 \\
      \hline
 & 1675692149588701583556012593441556134169533164932596422419308648533776178 \\
6 & 12160207490958418193300962 \\
 & -\,3413118505805601540291498 \\
      \hline
 & 3266268236007269922068112767862912834922718391 \\
7 & 1589814414379364704593946359 \\
 & -\,52202574090563189329673 \\
      \hline
 & 217064892339276374896058217864319147819031514610535561445877911458269709445733 \\
8 & 1773286019674481300663325709801 \\
 & -\,425427118896332660731 \\
      \hline
      \hline
\end{tabular}
\end{scriptsize}
\end{center}
\end{table}

We continue our discussion of 1-dimensional families from section 4. It is clear from Lemma~\ref{prop1} that
$\l\in k(i_1, i_2) $ is a rational function in terms of $i_1$ and $i_2$. Thus, for every nonsingular moduli point
$\p =(i_1, i_2)$ we get $\l \in M(\p)$. Hence, the equation of the hyperelliptic curve as in Eq.~\eqref{eq_1_dim}
is a rational model over the field of moduli.

However, on the singular points of the curve $F(i_1, i_2)=0$ direct computation for $\l$ is needed. In all the
cases the singular points have rational coordinates in the curve $F(i_1, i_2)=0$. However, this is not sufficient
for the moduli point to be a rational point. For each point, let $k (\sqrt{d}) $ denote the quadratic extension
determined by the corresponding polynomial of Table 2. From the corresponding values of $\l$ we compute the $i_3$
invariant. In all the cases this is not $k$-rational and $i_3 \in k (\sqrt{d})$. Hence, the field of moduli
contains $k (\sqrt{d})$. Since the curve has equation given in Eq.~\eqref{norm_eq} then $k (\sqrt{d}) $ is a field
of definition. Hence, $k (\sqrt{d}) $ is the field of moduli and Eq.~\eqref{norm_eq} provides a rational model
over this field. These computations are summarized in Table 3, where $d$ is determined for all singular points.

\section*{Acknowledgements}
Most of this paper was written during a visit of the first author at the University of Idaho.  The first author
wants to thank T. Shaska for his invitation. Further, he wants to thank the Spanish Ministry of Education for the
grant support in Project BFM2001-1294. Both authors want to thank the anonymous referees for helpful comments.



\begin{thebibliography}{99}
\bibitem {BS} \textsc{ R. Brandt, H. Stichtenoth},
Die Automorphismengruppen hyperelliptischer Kurven. Manuscripta Math.  55  (1986),  no. 1, 83--92.


\bibitem {BCGG} \textsc{E. Bujalance, F. J. Cirre, J. M. Gamboa and G. Gromadzki},
Symmetry types of hyperelliptic Riemann surfaces, Mém. Soc. Math. Fr. No. 86 (2001).

\bibitem {Cl} \textsc{A.  Clebsch},
Theorie der  Bin\"aren Algebraischen Formen, Verlag von B.G. Teubner, Leipzig (1872).

\bibitem {Gu} \textsc{J. Gutierrez}, A polynomial decomposition algorithm
over factorial domains, Comptes Rendues Mathematiques, de Ac. de Sciences, 13 (1991), 81-86.

\bibitem {Gu2} \textsc{J. Gutierrez, R. Rubio and D. Sevilla},
On multivariate rational function decomposition. \emph{Journal of Symb. Comput.} Vol. 33 (5), 546-562 (2002).

\bibitem {GS} \textsc{J. Gutierrez and T. Shaska},
Hyperelliptic curves with extra involutions,  LMS JCM, (8)   102-115, 2005.

\bibitem {GSS} \textsc{J. Gutierrez, D. Sevilla, and T. Shaska},
Hyperelliptic curves of genus 3 and their automorphisms, Lect. Notes in Comp, vol 13. (2005), 109--123.

\bibitem {Kl} \textsc{F. Klein}, Lectures on the Icosahedron and the Solution of Equations of the Fifth
Degree. Dover Publications, Inc., New York, N. Y., 1956.

\bibitem {Sh1} \textsc{T. Shaska},
Some special families of hyperelliptic curves, {\it J. Algebra Appl.}, vol \textbf{3}, No. 1 (2004), 75-89.

\bibitem {Sh5}  \textsc{T. Shaska},
Computational aspects of hyperelliptic curves, Computer mathematics. Proceedings of the sixth Asian symposium
(ASCM 2003), Beijing, China, April 17-19, 2003. River Edge, NJ: World Scientific. \emph{Lect. Notes Ser. Comput.}
10, 248-257 (2003).

\bibitem {H} \textsc{H. V\"olklein},
Groups as Galois groups. An introduction. Cambridge Studies in Advanced Mathematics, 53. Cambridge University
Press, Cambridge, 1996.


\end{thebibliography}
\end{document}